\documentclass[a4paper,12pt]{article}

\usepackage{amsmath}
\usepackage{amssymb}
\usepackage{amsthm}
\usepackage{amsfonts}
\usepackage{graphicx}
\usepackage{rotating}
\usepackage{setspace}
\usepackage{url}
\usepackage{epsfig,tikz,multirow}
\usepackage{epstopdf}
\usepackage{subfig}
\usepackage{color}
\usepackage{leftidx}
\usepackage{authblk}
\usepackage{float}

\newtheorem{theorem}{Theorem}[section]

\newtheorem{remark}{Remark}[section]

\newtheorem{lemma}{Lemma}[section]

\theoremstyle{definition}

\newtheorem{example}{Example}[section]

\numberwithin{equation}{section}

\DeclareMathOperator{\im}{Im}

	\title{Existence of periodic solution of a non-autonomous allelopathic phytoplankton model with fear effect}
		\author[1]{Satyam Narayan Srivastava}
		\author[1]{Alexander Domoshnitsky}
		\author[2]{Seshadev Padhi}
		\author[3]{Rana D. Parshad}
		\affil[1]{Department of Mathematics, Ariel University, Ariel-40700, Israel, satyamsrivastava983@gmail.com, adom@ariel.ac.il}
		\affil[2]{Department of Mathematics, Birla Institute of Technology, Ranchi-835215, India, spadhi@bitmesra.ac.in}
		\affil[3]{Department of Mathematics, Iowa State University, Ames, IA 50011, USA, rparshad@iastate.edu}
		
	\date{}
	
\begin{document}
		\maketitle

		\begin{abstract} 
In this paper, we consider a non-autonomous allelopathic phytoplankton competition ODE model, incorporating the influence of fear effects observed in natural biological phenomena. Based on Mawhin’s coincidence degree theory some sufficient conditions for existence of periodic solutions are obtained. We validate our findings through an illustrative example and numerical simulations, showing that constant coefficients lead to steady-state dynamics, while periodic variations induce oscillatory behavior.
%
%
			\smallskip
			
			{\it Key Words and Phrases}:  allelopathy, phytoplankton, Competition, fear effect, existence of solution, periodicity, coincidence degree theory
			
		\end{abstract}
		
	
	\section{Introduction}
	Phytoplankton are the autotrophic components of the plankton community and a key part of ocean and freshwater ecosystems. Phytoplankton form the base of aquatic food webs, are crucial for ecosystem functions and services, and significantly benefit the biotechnology, carbon sequestration pharmaceutical, and nutraceutical sectors \cite{briggs2020major,pradhan2022phytoplankton, winder2012phytoplankton}. A distinct phenomenon observed among phytoplankton species is the production of secondary metabolites by one species that inhibit the growth or physiological functions of another phytoplankton species \cite{ccsp22}. 
	This behavior, known as allelopathy, occurs when phytoplankton engage in competitive interactions by releasing toxic compounds. Numerous studies have demonstrated that allelopathy plays a vital role in shaping the competitive dynamics among phytoplankton. For example, Maynard-Smith \cite{maynard1974models} incorporated an allelopathic term into the classical two-species Lotka–Volterra competition model to account for the harmful effects one species exerts on the other:
	\begin{equation*}
	\begin{cases}
		\frac{dN_1(t)}{dt} = N_1(t) \left( \alpha_1 - \beta_1 N_1(t) - v_1 N_2(t) - \gamma_1 N_1(t) N_2(t) \right), \\
		\frac{dN_2(t)}{dt} = N_2(t) \left( \alpha_2 - \beta_2 N_2(t) - v_2 N_1(t) - \gamma_2 N_1(t) N_2(t) \right),
	\end{cases}
	\end{equation*}
	where  $N_i(t)$ $(i = 1, 2)$ represents the density of the two competing phytoplankton species, \( \alpha_i \) is the daily cell proliferation rate, $\beta_i$ denotes the intraspecific competition rate of the $i$-th species, $v_i$ represents the interspecific competition rate, and  $\gamma_i$ is the toxicity coefficient exerted by the other species on species $i$. The initial conditions are $N_i(0) > 0$.

	\ \ \ \ \ Building on the work of Maynard-Smith, numerous researchers have examined scenarios where only one species releases toxins. Chen et al.\ \cite{insp1} proposed a discrete system to model toxin release by a single species:
\begin{equation*}
	\begin{cases}
			x_1(n+1) &= x_1(n) \exp \left[ r_1(n) - a_{11}(n)x_1(n) - a_{12}(n)x_2(n) - b_1(n)x_1(n)x_2(n) \right], \\
		x_2(n+1) &= x_2(n) \exp \left[ r_2(n) - a_{21}(n)x_1(n) - a_{22}(n)x_2(n) \right].
	\end{cases}
\end{equation*}	
The authors established the conditions for extinction and global stability of system. It was demonstrated that at low rates of toxin release, the extinction dynamics of system remain unaffected, indicating that the toxic species cannot drive the non-toxic species to extinction. Further investigations into allelopathic \cite{chen2013extinction,insp2,mandal2023toxicity}

\ \ \ \  In reality, a non-toxic species can go extinct even when exposed to only low concentrations of toxins. This raises the question of what other factors, aside from degradation caused by toxins, might influence the density of competing phytoplankton species, in the absence of additional external influences. Given that allelopathy is modeled on the classical Lotka–Volterra competition framework, we will explore the concept of competitive fear.

\ \ \ \ Wang et al.\ \cite{wang2016modelling} were the first to incorporate the fear effect into the classical two-species Lotka–Volterra predator–prey model. They defined the fear effect function as $f(k, y) = \frac{1}{1 + ky}$, which represents the prey's anti-predation response induced by fear of the predator. The study revealed that under conditions of Hopf bifurcation, an increase in the fear level could shift the direction of the Hopf bifurcation from supercritical to subcritical, provided the prey's birth rate also increases. Numerical simulations further indicated that anti-predator defenses of prey intensify as the predator's attack rate increases. Additional studies on the fear effect in predator–prey models can be found in \cite{biswas2021delay,kaur2021impact,lai2020stability,liu2022stability}.

\ \ \ \ \ The impact of fear on predator–prey systems has been widely explored; however, its role in competition systems has received significantly less attention. Nevertheless, compelling evidence suggests that fear can exist in purely competitive systems, even in the absence of predation effects or when such effects are negligible \cite{ccsp48}.

\ \ \ \ \ One well-known example of ``fear" as a factor in a predator-prey relationship is between wolves (Canis lupus) and elk (Cervus canadensis) in Yellowstone National Park, USA \cite{laundre2001wolves}. This relationship illustrates how the fear of predation can significantly impact prey behavior and even ecosystem structure, beyond just the direct impact of being hunted. After wolves were reintroduced to Yellowstone in 1995, elk populations had to adapt not only to the risk of predation but also to the heightened stress and vigilance required to avoid wolves. Studies have shown that elk alter their grazing patterns, avoiding open areas where wolves are more likely to spot them, and spend more time in dense forest cover. This behavioral shift reduced the browsing pressure on young aspen and willow trees, which allowed these plants to grow taller and even led to a resurgence of certain riparian ecosystems that depend on these trees for structural stability and habitat landscape of fear" effect, where prey animals modify their behavior due to the perceived risk of predation, exemplifies how the presence of predators can influence ecosystems by inducing fear in prey species, even if actual predation rates remain low.

\ \ \ \ The work of Srivastava et al.\ \cite{fearcm} is the first to consider fear in competitive systems, to the best of our knowledge. Inspired by the works \cite{insp1,insp4,kaur2021impact,insp5,fearcm}, Chen et.\ al \cite{apm2} investigate how the fear parameter affects competitive allelopathic planktonic systems by introducing a fear effect term, where the non-toxic species is ``fearful" of the toxic species. Authors in \cite{apm2} performed dynamical analysis on the model
	 \begin{align}
	 	\frac{d x_{1}}{d \tau} =& r_{1}x_{1} \left( 1- \frac{x_{1}}{k_{1}}\right)- \beta_{1}x_{1} x_{2},   \label{e1} \\
	 	\frac{d x_{2}}{d \tau} =&r_{2} x_{2} \left( \frac{1}{1+ w_{1} x_{1}} - \frac{x_{2}}{k_{2}}\right) - \beta_{2}x_{1}x_{2} - w_{2}x_{1} x_{2}^{2} \label{e2}
	 \end{align}
where $r_{i}$,  $k_{i}$, $\beta_{i}$ $(i=1,2)$, $w_{1}$, and $w_{2}$, respectively, denote the intrinsic growth rate, carrying capacity, the interspecific competition rate, the fear effect parameter, and the toxicity coefficient.  

\ \ \ To more accurately reflect the real-world influences of climate change and seasonal variations, many researchers have studied non-autonomous models, examining their dynamic behavior, including permanence, as well as the existence and stability of positive periodic solutions \cite{apm4,apm5,apm6,apm7,apm3,apm8}. In \cite{abbas2012almost}, authors obtained the sufficient conditions for permanence along with existence-uniqueness of an almost periodic solution of a non-autonomous two species competitive allelopathic phytoplankton model in presence of a discrete time delay. In \cite{zhao2020extinction}, a non-autonomous allelopathic phytoplankton model with nonlinear inter-inhibition terms and feedback controls is studied. However, no study has yet explored the non-autonomous allelopathic phytoplankton model with fear effect. A non-autonomous model can effectively incorporates these time-dependent effects. 

\ \ \ Motivated by the above mentioned applications and theoretical investigations, in this paper, we formulate a non-autonomous allelopathic phytoplankton model with fear effect and investigate the existence of its positive periodic solution. The model is constructed from \eqref{e1}-\eqref{e2} using non-autonomous parameters
	 \begin{align}
	\frac{d x_{1}}{d \tau} =& r_{1}(t)x_{1}(t) \left( 1- \frac{x_{1}(t)}{k_{1}}\right)- \beta_{1}(t)x_{1}(t) x_{2}(t),   \label{e3} \\
	\frac{d x_{2}}{d \tau} =&r_{2}(t) x_{2}(t) \left( \frac{1}{1+ w_{1}(t) x_{1}(t)} - \frac{x_{2}(t)}{k_{2}}\right) - \beta_{2}(t)x_{1}(t)x_{2}(t) - w_{2}x_{1}(t) x_{2}^{2} (t) \label{e4}
\end{align}
where $r_{1}(t)$, $r_{2}(t)$, $\beta_{1}(t)$, and $\beta_{2}(t)$ are all positive $T$-periodic functions.The ecological
interpretation of the parameters in model \eqref{e3}-\eqref{e4} remains the same as in model \eqref{e1}-\eqref{e2} where $x(t)$ and $y(t)$ represent the toxic and non-toxic species, respectively, and the
other parameters also retain their ecological meaning as described previously.

\ \ \ Many mathematical models are formulated using differential equations or systems of differential equations. The expanding diversity of differential equation systems has attracted numerous researchers to explore the dynamical behavior of solutions, with a particular focus on the existence and uniqueness of solutions. Several studies in the literature employ
Schauder’s fixed point theorem, Krasnosel’skii’s fixed point theorem, and Mawhin’s coincidence degree theory to investigate the existence of positive periodic solutions in a variety of biological models. In recent years, the coincidence degree method has become a powerful and effective tool for addressing the existence of periodic solutions in both differential and
difference equations. Recent studies employing coincidence degree theory for various biological models include \cite{bohner2006existence,apm4,apm5,apm6,li2018existence,apm7,apm3,apm8}. The present paper aims to find some suitable conditions for the existence of positive periodic solution for system \eqref{e3}-\eqref{e4}.

\ \ \ \ The remainder of this paper is organized as follows: Section 2 introduces the fundamental concepts of coincidence degree theory and some useful notations. In Section 3, we establish sufficient conditions for the existence of a positive periodic solution using the continuation theorem of coincidence degree theory and provide an illustrative example. Finally, in Section 4, we summarize our findings and suggest potential avenues for future research.

\section{Preliminaries}

Before presenting our results on the existence of periodic solution of system \eqref{e3}-\eqref{e4} we
briefly introduce the coincidence degree theory and some notations as follows
 
 \ \ \ Let $Z$ and $W$ be the real Banach spaces, and Let $L:dom (L)\subset Z \rightarrow W$ be Fredholm operator of index zero, If $P:Z\rightarrow Z$ and $Q:W\rightarrow W$ are two continuous projectors such that $\im (P)= Ker (L),$ $Ker(Q)=Im(L),$ $Z=Ker(L)\oplus Ker(P)$ and $W=Im(L)\oplus Im(Q),$ then the inverse operator of $L|_{dom(L)\cap Ker(P)}: dom(L)\cap Ker(P)\rightarrow Im(L)$ exists and is denoted by $K_{p}$ (generalized inverse operator of $L$). If $\Omega$ is an open bounded subset of $Z$ such that $dom(L)\cap \Omega \neq 0$, the mapping $N:Z\rightarrow W$ will be called L-compact on $\overline{\Omega}$, if $QN(\overline{\Omega})$ is bounded and $K_{p}(I-Q)N:\overline{\Omega} \rightarrow Z$ is compact. 	The abstract equation $Lx=Nx$ is shown to be solvable in view of \cite{mawhin}.

\begin{theorem}[\cite{mawhin}]\label{thm2.1}
	Let $L$ be a Fredholm operator of index zero and let $N$ be the L-compact on $\overline{\Omega}$.
	Assume the following conditions are satisfied:
	\begin{enumerate}
		\item[1)] $Lx\neq \lambda Nx$ for every $(x,\lambda)\in [(dom(L) \backslash Ker(L))\cap\partial\Omega] \times (0,1)$;
		\item[2)] $Nx \notin Im(L)$ for every $x\in Ker(L)\cap \partial \Omega$;
		\item[3)] $deg(QN|_{Ker(L)}, Ker(L)\cap \Omega,0)\neq 0$, where $Q:W\rightarrow W$ is a projector as above with $Im(L)=Ker(Q)$.  
	\end{enumerate}
	Then, the equation $Lx=Nx$ has at least one solution in $dom(L)\cap \overline{\Omega}$.
\end{theorem}

For convenience, we shall introduce the notations
\begin{center}
	$\overline{f}=\frac{1}{T} \int_{0}^{T} f(t)dt$, \ \ \ $f^{L}=\min\limits_{t\in[0,1]}f(t)$, \ \ \ $f^{M}=\max\limits_{t\in[0,1]}f(t)$,
\end{center} 
where $f$ is a continuous $T$-Periodic function.

Set:
\begin{equation*}
	m_{\epsilon}= \frac{k_{2}r_{2}^{M}}{1+w_{1}k_{1}}+\epsilon, \ \ \ g_{\epsilon}=\frac{k_{1}\beta_{1}^{M}m_{0}}{r_{1}^{M}} + \epsilon, \ \ \ \  h_{\epsilon}=\frac{r_{2}^{L}k_{2}}{(1+w_{1}k_{1}) (r_{2}^{M} + w_{2}k_{2}k_{1})}.
\end{equation*}
Also, there exist positive numbers $L_{i} \ \ (i=1,2,\ldots,4)$ such that $L_{2}\leq z_{1}(t)\leq L_{1}$, $L_{4}\leq z_{2} (t) \leq L_{3}$, where $L_{i} \ (i=1,2\ldots,4)$ will be calculated as in the proof of following Theorem.

\section{Existence of the Periodic Solution}

In order to apply Mawhin coincidence degree theory, we need to make a change of variables.
Consider  $$z_{1}(t)=\ln x_{1}(t) \Rightarrow x_{1}(t)=e^{z_{1}(t)},$$
$$z_{2}(t)=\ln x_{2}(t) \Rightarrow x_{2}(t)=e^{z_{2}(t)},$$

then system \eqref{e3}-\eqref{e4} becomes
\begin{align}
	\frac{d z_{1}}{d \tau} =& r_{1}(t)\left( 1- \frac{e^{z_{1}(t)}}{k_{1}}\right)- \beta_{1}(t) e^{z_{2}(t)}, \label{e3.1} \\
	\frac{d x_{2}}{d \tau} =& r_{2}(t)  \left( \frac{1}{1+ w_{1} e^{z_{1}(t)}} - \frac{e^{z_{2}(t)}}{k_{2}}\right) - \beta_{2}(t)e^{z_{1}(t)} - w_{2}e^{z_{1}(t)} e^{z_{2}(t)}. \label{e3.2}
\end{align}

The spaces $Z$ and $W$ are both Banach spaces, defined as follows
\begin{equation*}
	Z = W = \{ z = (z_1, z_2) \in (\mathbb{R}, \mathbb{R}^2) \mid z(t+T) = z(t) \}
\end{equation*}
where each element $z = (z_1, z_2)$ satisfies the periodic condition $z(t+T) = z(t)$. These spaces are equipped with the norm $\| \cdot \|$ given by
\begin{equation*}
	\|z\| = \max_{t \in [0,T]} \sum_{i=1}^{2} |z_i|, \quad z = (z_1, z_2) \in Z \text{ or } W.
\end{equation*}
This norm ensures that both $Z$ and $W$ are complete, making them Banach spaces.

Define operators $L$, $P$, and $Q$ as follows, respectively:
\begin{equation*}
	L : \text{dom}(L) \cap Z \to W, \quad Lz = \left( \frac{dz_1}{dt}, \frac{dz_2}{dt} \right)
\end{equation*}
\begin{equation*}
	P\begin{pmatrix} z_1 \\ z_2 \end{pmatrix} = Q\begin{pmatrix} z_1 \\ z_2 \end{pmatrix} =
	\begin{pmatrix} \frac{1}{T} \int_0^T z_1(t) dt \\ \frac{1}{T} \int_0^T z_2(t) dt \end{pmatrix},
	\quad \begin{pmatrix} z_1 \\ z_2 \end{pmatrix} \in Z = W,
\end{equation*}
where
\begin{equation*}
	\text{dom}(L) = \{ z \in Z : z(t) \in C^1(\mathbb{R}, \mathbb{R}^2) \}.
\end{equation*}

We now define $N : Z \times [0,1] \to W$ as
\begin{equation*}
	N\begin{pmatrix} z_1 \\ z_2 \end{pmatrix} =
	\begin{pmatrix} \Gamma_1(z,t) \\ \Gamma_2(z,t) \end{pmatrix}.
\end{equation*}
where
\begin{equation*}
	\Gamma_1(z,t) = r_1(t) \left( 1 - \frac{e^{z_{1}(t)}}{k_{1}} \right) - \beta_1(t)e^{z_2(t)},
\end{equation*}
\begin{equation*}
	\Gamma_2(z,t) = r_2(t) \left( \frac{1}{1 + w_1 e^{z_1(t)}} - \frac{e^{z_{2}(t)}}{k_{2}}  \right) - \beta_2(t)e^{z_1(t)} - w_2 e^{z_1(t)} e^{z_2(t)}.
\end{equation*}
These functions are $T$-periodic. In fact,
\begin{equation*}
	\Gamma_1(z(t+T), t+T) = r_1(t) \left( 1 - \frac{e^{z_{1}(t)}}{k_{1}}\right) - \beta_1(t)e^{z_2(t)}.
\end{equation*}
Clearly, $\Gamma_2(z,t)$ is also a periodic function, as can be shown in a similar manner.

\  \ \ \ We introduced a change of variables, defined the relevant spaces, norms, and operators. To satisfy the conditions of Theorem \ref{thm2.1}, we first present two lemmas and then establish a theorem ensuring the existence of a periodic solution based on these lemmas.
\begin{lemma}\label{lem3.1}
	$L$ is a Fredholm operator of index zero and $N$ is $L$-compact.
\end{lemma}
\begin{proof}
		It is easy to observe that $$Ker(L)= \{ z \in Z| z=c_{0}, \ c_{0} \in \mathbb{R}^{2} \},$$ and $$Im(L)=\{z\in W | \int_{0}^{T} z(t)dt=0 \}$$
	is closed in $W$. Furthermore, both $P$, $Q$ are continuous projections satisfying
	$$ Im(P)=Ker(L), \ Im(L)=Ker(Q)= Im(I-Q).$$
	
	For any $z \in W$, let $z_{1} = z - Qz$, we can obtain that
	$$\int_{0}^{T} z_{1}dp= \int_{0}^{T} z(p)dp - \int_{0}^{T} \frac{1}{T} \int_{0}^{T} z(t)dtdp=0,$$
	so $z_{1} \in Im(L)$. It follows that $W = Im(L) + Im(Q) = Im(L) +\mathbb{R}^{2}$. Since $Im(L) \cup \mathbb{R}^{2} = {0}$,
	we conclude that $W = Im(L) \oplus  \mathbb{R}^{2},$ which means $dim  \ Ker (L) = codim \  Im(L) = dim \ (\mathbb{R}^{2}) = 2$. Thus, $L$ is a Fredholm operator of index zero, which implies that $L$ has a unique generalized inverse operator.
	
	\ \ \ \ Next we show that $N$ is $L$-compact. Define the inverse of $L$ as $K_{P} : Im(L) \rightarrow \ Ker(P) \cap dom(L)$
	and is given by
	$$ K_{P}(z)= \int_{0}^{t} z(s)ds- \frac{1}{T} \int_{0}^{T} \int_{0}^{t}z(s)dsdt.$$
	
	Therefore, for any $z(t) \in Z$, we have
	\[QN\left(
	\begin{array}{c}
		z_{1} \\
		z_{2} \\
	\end{array}
	\right)
	= \left(
	\begin{array}{c}
		\frac{1}{T} \int_{0}^{T}	\Gamma_{1}(z,t)dt \\
		\frac{1}{T} \int_{0}^{T}	\Gamma_{2}(z,t)dt \\
	\end{array}
	\right),  \]
	and
	\begin{align*}
		K_{P} (I-Q)Nz &= \int_{0}^{t} Nz(s)ds - \frac{1}{T} \int_{0}^{T}\int_{0}^{t} Nz(s)dsdt - \frac{1}{T} \int_{0}^{t} \int_{0}^{T} QNz(s)dtds \\
		& \ \ \ \ \ \ \ \ \ \ \ \ \ \ \ \ \ \ \ \ \ \ \ \ \ \ \ \ \  \ \ \ \ \ \ \ \ \ \ \ \ \ \ + \frac{1}{T^{2}} \int_{0}^{T} \int_{0}^{t} \int_{0}^{T} QNz(s) dt ds dt  \\
		&= \int_{0}^{t} Nz(s)ds - \frac{1}{T} \int_{0}^{T}\int_{0}^{t} Nz(s)dsdt - \left(\frac{t}{T}-\frac{1}{2}\right) \int_{0}^{T} QNz(s)ds.
	\end{align*}
	Clearly, $QN$ and $K_{P}(I-Q)N$ are continuous. Due to $Z$ is Banach space, using Arzela-Ascoli theorem, we have that $N$ is $L$-compact on $\overline{U}$ for any open bounded set $U \subset Z.$ 
\end{proof}

\begin{lemma}\label{lem3.2}
	For $\lambda \in (0,1)$, we examine the following family of systems
	\begin{equation}\label{e3.3}
		\begin{cases}
			\frac{dz_1}{dt} = \lambda \left[ r_1(t) \left( 1 - \frac{e^{z_1(t)}}{k_1} \right) - \beta_1(t)e^{z_2(t)} \right], \\
			\frac{dz_2}{dt} = \lambda \left[ r_2(t) \left( \frac{1}{1+w_1e^{z_1(t)}} - \frac{e^{z_2(t)}}{k_2} \right) - \beta_2(t)e^{z_1(t)} - w_2e^{z_1(t)}e^{z_2(t)} \right].
		\end{cases}
	\end{equation}
	
	If the following conditions hold:
	\begin{itemize}
\item[(A1)] $k_2 r_2^M < 1 + w_1 k_1$, 
\item[(A1)]$\beta_1^M m_0 < r_1^L$, 
\item[(A1)] $ r_2^L k_2 < (1 + w_1 k_1)(r_2^M + w_2 k_2 k_1)$,
	\end{itemize}
	
	then any periodic solution $(x_1, x_2)$ satisfies the existence of positive constants $L_i$ (for $i = 1,2,\dots,4$) such that:
	\begin{equation*}
		L_2 \leq x_1(t) \leq L_1, \quad L_4 \leq x_2(t) \leq L_3,
	\end{equation*}
	where the values of $L_i$ (for $i = 1,2,\dots,4$) will be computed in the proof.
\end{lemma}
\begin{proof}
	We assume that $z \in (z_{1},z_{2})^{T} \in Z$ is a $T$-periodic solution of system \eqref{e3.1}-\eqref{e3.2} for any fixed $\lambda \in (0,1).$  Since $(z_{1},z_{2}) \in Z,$ there exist $\eta_{i}, \xi_{i} \in [0,T]$ such that
$$ z_{i}(\eta_{i}) = \max\limits_{t\in[0,T]} z_{i}(t), \ \ \ \ z_{i}(\xi_{i}) = \min\limits_{t\in[0,T]} z_{i}(t), \ \ \ i=1,2. $$

Through simple analysis, we have,
\begin{equation*}
	\dot{z_{1}}(\eta_{1})=\dot{z_{1}}(\xi_{1})=0, \ \ \ \ \ \dot{z_{2}}(\eta_{2})=\dot{z_{2}}(\xi_{2})=0.
\end{equation*}

Aapplying the previous result to \eqref{e3.3}, we obtain
\begin{equation}\label{eq3.5}
	r_{1}(\eta_{1})\left( 1- \frac{e^{z_{1}(\eta_{1})}}{k_{1}}\right)- \beta_{1}(\eta_{1}) e^{z_{2}(\eta_{1})} =0,
\end{equation}
\begin{equation}\label{eq3.6}
	r_{2}(\eta_{2})  \left( \frac{1}{1+ w_{1} e^{z_{1}(\eta_{2})}} - \frac{e^{z_{2}(\eta_{2})}}{k_{2}}\right) - \beta_{2}(t)e^{z_{1}(\eta_{2})} - w_{2}e^{z_{1}(\eta_{2})} e^{z_{2}(\eta_{2})} =0,
\end{equation}
and
\begin{equation}\label{eq3.7}
	r_{1}(\xi_{1})\left( 1- \frac{e^{z_{1}(\xi_{1})}}{k_{1}}\right)- \beta_{1}(t) e^{z_{2}(\xi_{1})}=0,
\end{equation}
\begin{equation}\label{eq3.8}
	r_{2}(\xi_{2})  \left( \frac{1}{1+ w_{1} e^{z_{1}(\xi_{2})}} - \frac{e^{z_{2}(\xi_{2})}}{k_{2}}\right) - \beta_{2}(\xi_{2})e^{z_{1}(\xi_{2})} - w_{2}e^{z_{1}(\xi_{2})} e^{z_{2}(\xi_{2})} =0.
\end{equation}

From \eqref{eq3.5}, we obtain
\begin{equation*}
	r(\eta_{1})-\frac{r(\eta_{1})e^{z_{1}(\eta_{1})}}{k_{1}} >0,
\end{equation*}
which implies that
\begin{equation}\label{eq3.10}
	z_{1}(\eta_{1}) < \ln (k_{1}) = L_{1}.
\end{equation}

Considering \eqref{eq3.6} and \eqref{eq3.10}, we get
\begin{align*}
	\frac{r_{2}(\eta_{2})}{1+w_{1}e^{z_{1}(\eta_{2})}}-\frac{e^{z_{2}(\eta_{2})}}{k_{2}} > 0
\end{align*}
So, we can obtain 
\[ 		\frac{e^{z_{2}(\eta_{2})}}{k_{2}} < \frac{r_{2}(\eta_{2})}{1+w_{1} e^{z_{1}(\eta_{2})}} \]
or
\[ e^{z_{2}(\eta_{2})} < \frac{k_{2} r_{2}^{M}}{1+ w_{1} k_{1}} \]
which gives
\begin{align}\label{L3}
	z_{2}(\eta_{2})&<  \ln \left( \frac{k_{2}r_{2}^{M}}{1+w_{1}k_{1}}\right) = \ln m_{0} =L_{3}.	 
\end{align}

From \eqref{eq3.7} and \eqref{L3}, we can obtain
\begin{equation*}
	r_{1}(\xi_{1}) \left( 1- \frac{e^{z_{1}(\xi_{1})}}{k_{1}}\right) = \beta_{1}(\xi_{1}) e^{z_{2}(\xi_{1})}
\end{equation*}
then, \[ 1- \frac{e^{z_{1}(\xi_{1})}}{k_{1}} = \frac{\beta_{1}(\xi_{1}) e^{z_{2}(\xi_{1})}}{r_{1}(\xi_{1})}\]
or, \[ \frac{e^{z_{1}(\xi_{1})}}{k_{1}} >1- \frac{\beta_{1}^{M} m_{0}}{r_{1}^{L}}\]
which implies that
\begin{align}\label{L2}
	z_{1}(\xi_{1}) &> \ln \left( k_{1} \left(1-\frac{\beta_{1}^{M} m_{0}}{ r_{1}^{L}}\right)\right) = \ln (g_{0}) = L_{2}.
\end{align}		
In view of \eqref{eq3.8} and \eqref{L2}, we have
\begin{align*}
	r_{2}(\xi_{2})  \left( \frac{1}{1+ w_{1} e^{z_{1}(\xi_{2})}} - \frac{e^{z_{2}(\xi_{2})}}{k_{2}}\right) = \beta_{2}(\xi_{2})e^{z_{1}(\xi_{2})} + w_{2}e^{z_{1}(\xi_{2})} e^{z_{2}(\xi_{2})}
\end{align*}
or,
\[\frac{r_{2}(\xi_{2}) e^{z_{2}\xi_{2}}}{k_{2}} + w_{2} e^{z_{1}(\xi_{2}) e^{z_{2}(\xi_{2})}} = -\beta_{2}(\xi_{2}) e^{z_{1}(\xi_{2})} + \frac{r_{2}(\xi_{2})}{1+ w_{1}e^{z_{1}(\xi_{2})}}\]
Thus,
\[ 	e^{z_{2}(\xi_{2})} \left(\frac{r_{2}(\xi_{2})+w_{2}k_{2} e^{z_{1}(\xi_{2})}}{k_{2}} > \frac{r_{2}(\xi_{2})}{1+w_{1}e^{z_{1}(\xi_{1})}}\right),\]
or
\[ 	e^{z_{2}(\xi_{2})} > \frac{r_{2}^{L} k_{2}}{(1+w_{1}k_{1})(r_{2}^{M}+w_{2}k_{2}k_{1})} \]
that is	
\begin{align}\label{L4}
	z_{2}(\xi_{2}) &> \ln \left( \frac{r_{2}^{L} k_{2}}{(1+w_{1}k_{1})(r_{2}^{M}+w_{2}k_{2}k_{1})}  \right)  = \ln (h_{0}) = L_{4}.
\end{align}

Form \eqref{eq3.10}, \eqref{L3}, \eqref{L2}, \eqref{L4}, we get
$$ |z_{1}(t)| < max \{ |L_{1}|,|L_{2}|\} = \Lambda_{1},$$
$$ |z_{2}(t)| < max \{ |L_{3}|,|L_{4}|\} = \Lambda_{2}.$$
where $\Lambda_{1}, \Lambda_{2} $ is independent of $\lambda$.	Denote $\Lambda=\Lambda_	{1}+\Lambda_{2}+\Lambda_{3}$ where $\Lambda_{3}$ is taken sufficiently large such that each solution $ (z_{1}^{*},z_{2}^{*})$ of system
\begin{equation}\label{e12}
	\overline{r_{1}}- \frac{\overline{r_{1} e^{z_{1}(t)}}}{k1} - \overline{\beta_{1}} e^{z_{2}(t)} =0,
\end{equation}
\begin{equation}\label{e13}
	\frac{\overline{r_{2}}}{1+ w_{1} e^{z_{1}(t)}} -\frac{ \overline{r_{2}} e^{z_{2}(t)}  }{k_{2}} - \overline{\beta_{2}} e^{z_{1}(t)} - w_{2}e^{z_{1}(t)} e^{z_{2}(t)}=0,
\end{equation}
satisfies $|z_{1}^{*}|+|z_{2}^{*}|<\Lambda$. Now we consider $\Omega=\{(z_{1},z_{2})^{T} \in Z: \| (z_{1},z_{2})\| < \Lambda \}$ then it is clear that $\Omega$ satisfies the first condition of Theorem \ref{thm2.1}. 

\ \ \ We prove that $QN(z_{1},z_{2})^{T}\neq (0,0)^{T}$ for each $(z_{1},z_{2}) \in \partial\Omega \cap Ker(L)$. When $(z_{1},z_{2})^{T} \in \partial \Omega \cap Ker(L)=\partial \Omega \cap \mathbb{R}^{2}, \ (z_{1},z_{2})^{T}$ is a constant vector in $\mathbb{R}^{2}$ and $|z_{1}|+|z_{2}|=\Lambda$. If the system \eqref{e12}-\eqref{e13} has a solution, then
\[QN\left(
\begin{array}{c}
	z_{1} \\
	z_{2} \\
\end{array}
\right)
= \left(
\begin{array}{c}
	\overline{r_{1}}- \frac{\overline{r_{1} e^{z_{1}(t)}}}{k1} - \overline{\beta_{1}} e^{z_{2}(t)} \\
	\frac{\overline{r_{2}}}{1+ w_{1} e^{z_{1}(t)}} -\frac{ \overline{r_{2}} e^{z_{2}(t)}  }{k_{2}} - \overline{\beta_{2}} e^{z_{1}(t)} - w_{2}e^{z_{1}(t)} e^{z_{2}(t)}
\end{array}
\right)
\neq\left(
\begin{array}{c}
	0 \\
	0\\
\end{array}
\right). \]
Since, \eqref{e12}-\eqref{e13} does not have solution then, it is evident that $QN(z_{1},z_{2})^{T}\neq 0$, thus the second condition of the Theorem \ref{thm2.1} is satisfied.
\end{proof}

\ \ \ \ Having provided the necessary lemmas, we now present the existence of a positive periodic
solution for the system \eqref{e3}-\eqref{e4}.
\begin{theorem}\label{thm3.1}
	If the conditions (A1), (A2), and (A3) hold, then, system \eqref{e3}-\eqref{e4} has at
	least one positive T-periodic solution.
	\begin{proof}
	It follows from Lemma \ref{lem3.1} that $L$ is a Fredholm operator of index zero and $N$ is $L$-compact. By Lemma \ref{lem3.2}, we get that the first two conditions of Theorem \ref{thm2.1} are satisfied. Finally, we prove that the last condition of Theorem \ref{thm2.1} is satisfied, to do so, we define the following mapping 
		$ \Psi_{\mu} : dom (L) \times [0,1] \rightarrow Z$
		\[\Psi(z_{1},z_{2},\mu)
		= \left(
		\begin{array}{c}
			\overline{r_{1}}- \frac{\overline{r_{1} e^{z_{1}(t)}}}{k1} - \overline{\beta_{1}} e^{z_{2}(t)} \\
		 -\frac{ \overline{r_{2}} e^{z_{2}(t)}  }{k_{2}} - \overline{\beta_{2}} e^{z_{1}(t)} - w_{2}e^{z_{1}(t)} e^{z_{2}(t)}
		\end{array}
		\right)
		+
		\mu \left(
		\begin{array}{c} 0 \\
			\frac{\overline{r_{2}}}{1+ w_{1} e^{z_{1}(t)}} - w_{2}e^{z_{1}(t)} e^{z_{2}(t)}
		\end{array}
		\right),
		\]	
		where $\mu\in [0,1]$ is a parameter. When $(z_{1},z_{2})^{T} \in \partial\Omega \cap Ker(L)= \partial \Omega \cap R^{2}$, $(z_{1},z_{2})^{T}$ is a constant vector in $R^{2}$ with $\|(z_{1},z_{2})^{T}\|=W$. We will show that $(z_{1},z_{2})^{T}\in \partial\Omega \cap Ker(L)$, $\Psi((z_{1},z_{2})^{T},\mu)\neq 0$. The following algebraic equation
		\[\Psi(z_{1},z_{2},0)=0\]
		a unique solution
		$z_{1}^{*}=\ln \left(\frac{k_{1}}{\overline{r_{1}}} \left(\overline{r_{1} } - \frac{k_{1}k_{2}\overline{r_{1} } \overline{\beta_{1}} \overline{\beta_{2}}}{\overline{\beta_{1}} \overline{\beta_{2}}k_{1}k_{2}-\overline{r_{1}\overline{r_{2}}}}\right)\right)$ and $z_{2}^{*}=\ln \left(\frac{k_{1}k_{2}\overline{r_{1}\overline{\beta_{2}}}}{\overline{\beta_{1}} \overline{\beta_{2}} k_{1}k_{2} - \overline{r_{1}} \overline{r_{2}}} \right)$. Define the homomorphism $J: Im(Q) \to Ker(L)$, $Jz\equiv z$, a direct calculation shows that 
		\begin{align*}
			&deg( JQN(z_{1},z_{2})^{T}, \Omega \cap Ker(L), (0,0)^{T})	\\
			&\ \ \ =  deg( QN(z_{1},z_{2})^{T}, \Omega \cap Ker(L), (0,0)^{T})	 \\
			& \ \ \ = deg( \Psi(z_{1},z_{2},1)^{T}, \Omega \cap Ker(L), (0,0)^{T}) \\
			&\ \ \ = deg( \Psi(z_{1},z_{2},0)^{T}, \Omega \cap Ker(L), (0,0)^{T}) \neq 0 .
		\end{align*}	
		Therefore, we have verified all the requirements of coincidence degree theorem \ref{thm2.1} and the system \eqref{e3}-\eqref{e4} has at least one positive $T$-periodic solution.
		
		\ \ \  This completes the proof.	
	\end{proof}
\end{theorem}

\ \ \ We now present an illustrative example to demonstrate the applicability of our result.

\begin{example}\label{ex1}
Consider the system \eqref{e3}-\eqref{e4} with the following selected parameters:
	\[r_1(t) = 0.03+0.5 \sin t, \  r_2(t) = 0.0001+0.9 \  \sin t,
	\beta_1(t) = 0.0004+0.5 \  \sin t,  \beta_2(t) = 0.006+0.2  \ \sin t, \]
\[	k_1 = 8, \ k_2 = 6, \  w_1 = 20, \ w_2 = 2.\]
By substituting these values, system \eqref{e3}-\eqref{e4} transforms into the following form
\begin{equation}\label{e3.14}
	\begin{cases}
		\frac{dx_1}{dt} = (0.03 + 0.5 \sin t)x_1 \left( 1 - \frac{x_1}{8} \right) - (0.0004 + 0.5 \sin t)x_1 x_2, \\
		\frac{dx_2}{dt} = (0.0001 + 0.9 \sin t)x_2 \left( \frac{1}{1+20x_1} - \frac{x_2}{6} \right) - (0.006 + 0.2 \sin t)x_1 x_2 - 2x_1 x_2^2.
	\end{cases}
\end{equation}

To verify the conditions of Theorem \ref{thm3.1}, we compute
\[
	r_1^L  = 0.03, \  r_2^L = 0.0001, \  r_2^M = 0.9001, \  \beta_1^M = 0.5004.
\]

Next, we evaluate the key expression
\begin{equation*}
	m_0 = \frac{k_2 r_2^M}{1 + w_1 k_2} = \frac{6 \times 0.9001}{1 + 20 \times 6} = 0.0446.
\end{equation*}

Consequently, we obtain:
\[	k_2 r_2^M = 5.4006 < 1 + w_1 k_1 = 161, \]
\[	\beta_1^M m_0 = 0.0223 < r_1^L = 0.03, \]
and
\[	r_2^L k_2 = 0.0006 < (1 + w_1 k_1)(r_2^M + w_2 k_2 k_1) = 15600.9161.\]

It is evident that assumptions (A1), (A2), and (A3) are satisfied. Hence, according to Theorem \ref{thm3.1}, system \eqref{e3.14} has at least one positive $T$-periodic solution. The Figure \ref{f1} below shows the validity of our results.
	\begin{figure}[H]
	\begin{center}
		\includegraphics[width=0.9\textwidth]{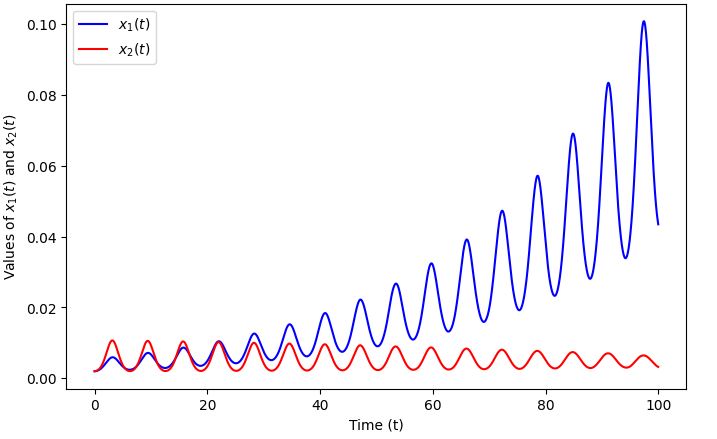} 
	\end{center}\caption{\label{f1}The periodic solution for system \eqref{e3.14} with initial condition $x_1(0) = 0.002$ and
	$x_2(0) = 0.002$. }	
\end{figure}
\end{example}
\begin{remark}
	\textnormal{It is important to highlight that in system \eqref{e3.14}, the presence of timedependent terms $\sin t$ introduces periodic fluctuations in the system’s behavior. However, if we remove the $\sin t$ terms, thereby making the coefficients purely constant as follows
		\[ r_1 = 0.03, \ r_2 = 0.0001, \  \beta_{1} = 0.0004, \  \beta_{2} = 0.006.\]
	In this case, the system exhibits steady-state dynamics rather than periodic behavior. This is clearly illustrated in Figure \ref{f2}, where the solutions stabilize over time instead of oscillating. This comparison highlights the crucial role of time-dependent fluctuations in generating periodic solutions, reinforcing the significance of our theoretical results. The presence of periodic coefficients directly influences the system’s behavior, leading to the emergence of periodic solutions, as demonstrated in Figure \ref{f1}. Thus, our findings validate the theoretical conditions established in Theorem \ref{thm3.1} and provide further insight into the impact of temporal variations on the system’s dynamics.}
\end{remark}
	\begin{figure}[H]
	\begin{center}
		\includegraphics[width=0.9\textwidth]{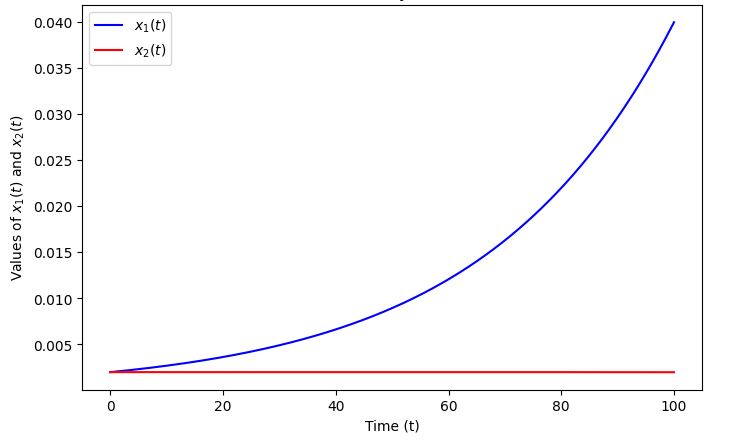} 
	\end{center}\caption{\label{f2} Steady dynamics for system (3.14) for constant coefficient with initial condition $x_1(0) = 0.002$ and
		$x_2(0) = 0.002$. }	
\end{figure}

\section{Discussion}
In this paper, we are concerned with competitive allelopathic planktonic systems with fear
effect, where the non-toxic species is “fearful” of the toxic species. The existence of positive periodic solutions is both significant and intriguing. We utilized Mawhin’s coincidence
degree theory to establish sufficient conditions for their existence. In nature, biological and
environmental parameters inherently vary over time. To capture this dynamic behavior in
our model, we introduced periodically varying coefficient functions, leading to periodic solutions for system \eqref{e3}-\eqref{e4}. We demonstrated our findings through an illustrative example
supported by numerical simulations. To further validate our results, we also considered a
case where the coefficients were taken as constants instead of periodic functions. In this
scenario, the system exhibited steady-state dynamics, highlighting the crucial role of timedependent variations in generating periodic behavior. The existence of periodic solutions
can offer insights into the seasonal dynamics of toxic planktonic species (such as harmful
algal blooms), such as their emergence, persistence, and decline, which is crucial for managing aquatic ecosystems. This understanding could aid in developing more effective control
strategies and improve the overall management of marine and freshwater ecosystems.

\ \ \ \  A recent study \cite{apm1} used coincidence degree theory to analyze dynamic equations on time
scales with multiple and time-varying delays, focusing on interacting phytoplankton species
that produce mutual toxins. Extending system \eqref{e3}-\eqref{e4} to a time scale with similar delays
and applying coincidence degree theory could be a promising direction for future research. Recent studies \cite{frac1,frac2} show that fractional calculus better captures complex systems and
natural phenomena than traditional integer-order models. Kunawat et al.\ \cite{kumawat2024mathematical} explored allelopathic stimulatory phytoplankton species using fractal–fractional derivatives, suggesting
an intriguing avenue for studying the fractional counterpart of the allelopathic phytoplankton
model with fear effects.

	\bibliographystyle{plain}
	\bibliography{Ref}
\end{document}